\documentclass[11pt]{article}
\usepackage[english]{babel}
\usepackage{amsmath}
\usepackage{amsthm}
\usepackage{amsfonts}
\usepackage{amssymb}
\usepackage{euscript}
\usepackage[cp1251]{inputenc}
\usepackage[pdftex]{graphicx}
\usepackage [autostyle, english = american]{csquotes}
\usepackage{tikz-cd}
\usepackage{float}
\usepackage[pdfpagelabels,bookmarksdepth=3]{hyperref}
\hypersetup{colorlinks=true,citecolor=black,urlcolor=black,linkcolor=black}

\newtheorem{theorem}{Theorem}[section]
\newtheorem{lemma}[theorem]{Lemma}

\theoremstyle{definition}

\theoremstyle{remark}

\numberwithin{equation}{section}

\begin{document}

\begin{center}
\textbf{\large{Noncommutativity of monotone Lagrangian cobordisms}}
\end{center}
\begin{center}
\textbf{Vardan Oganesyan}\\
University of California Santa Cruz,\\
E-mail address: vardanmath@gmail.com
\end{center}

\textbf{Abstract.} We construct a monotone spin Lagrangian cobordism from $L$ to $(L_1, L_2)$ such that there is no monotone spin Lagrangian cobordism from $L$ to $(L_2, L_1)$, where $L, L_1, L_2  \subset \mathbb{C}P^7$.

\section{Introduction}

Lagrangian cobordism is a relation between Lagrangian submanifolds of a symplectic manifold $(M, \omega)$ and was introduced by Arnold in \cite{Arnold}. First, we need to fix a class of Lagrangians we consider.

\vspace{0.04in}

Consider the set of immersed Lagrangians $\mathcal{L}ag(M)$. Let $\mathbb{Z}_2[\mathcal{L}ag(M)]$ be a set of finite formal linear combinations, where $\mathbb{Z}_2$ is $\mathbb{Z}/2\mathbb{Z}$. We define an abelian group
\begin{equation*}
\begin{gathered}
\mathcal{C}ob(M) = \mathbb{Z}_2[\mathcal{L}ag(M)]/\sim, \;\;\; \text{where $L_1 + \ldots + L_k \; \sim \; L_1' + \ldots + L_m' \;$ if} \\
 \text{$(L_1, \ldots, L_k)$ is cobordant to $(L_1', \ldots, L_m')$}.
\end{gathered}
\end{equation*}
This group was computed by Arnold for $M = \mathbb{C}$ and $M = T^{*}S^1$ in \cite{Arnold}. Note that the definition of cobordism groups given by Arnold is a little bit different. For exact symplectic manifold $M$ the study of the group $\mathcal{C}ob(M)$ was reduced to pure topological problem by Eliasberg in \cite{Eliashberg}. However, it is  hard to compute the group $\mathcal{C}ob(M)$ explicitly. The cobordism group for $\mathbb{C}^n$ and for cotangent bundles were studied in more details by Audin in \cite{Audin1, Audin2, Audin3}.

The study of $\mathcal{C}ob(M)$ was reduced to pure topological problem for arbitrary $M$ by Rathel-Fournier  in \cite{Dominique1, Dominique2}. Moreover, in these papers the groups $\mathcal{C}ob(S_g)$ were computed for $g \geqslant 2$, where $S_g$ is a surface of genus $g$.

In the definition above we can replace immersed Lagrangians by embedded Lagrangians and immersed cobordisms by embedded cobordisms. The study of embedded cobordisms can also be reduced to topological problems. If there is an immersed Lagrangian cobordism between embedded Lagrangians $(L_1, \ldots, L_k)$ and $(L_1', \ldots, L_k')$, then there is embedded Lagrangian cobordisms between them (see \cite[Secton 2.3]{Octav1}).

\vspace{0.04in}

In contrast, cobordisms satisfying suitable geometric constraints display remarkable rigidity phenomena and their study can not be reduced to pure topology.

Chekanov in \cite{Chekanov} showed that monotone cobordant Lagrangians have the same number (counted with signs) of pseudoholomorphic discs of Maslov index $2$ passing through a generic point.

It was shown by Biran and Cornea in \cite{Octav1, BC2, BC3, BC4} that monotone Lagrangian cobordisms are closely related to the derived Fukaya category of $M$. Similar results were obtained by Nadler and Tanaka in \cite{NT}.

\vspace{0.1in}

In this paper we found a new rigidity phenomena. Let $\mathcal{L}ag^{s,m}(M)$ be the set of spin and monotone Lagrangians. We construct Lagrangians $L, L_1, L_2 \in \mathcal{L}ag^{s,m}(M)$ and a spin monotone Lagrangian cobordism between $L$ and $(L_1, L_2)$. Then, we prove that there is no spin monotone Lagrangian cobordism between $L$ and $(L_2, L_1)$. Note that there is embedded Lagrangian cobordism between $L$ and $(L_2, L_1)$, but this cobordism can not be spin and monotone.

Let $F\mathcal{L}ag^{s,m}(M)$ be a free group generated by spin monotone Lagrangians. Consider a set of relations $R$
\begin{equation*}
\begin{gathered}
\text{ $L_1 \ldots L_k \; = \; L_1' \ldots L_m'$ \quad if } \\
\text{ $(L_1, \ldots, L_k)$ is spin monotone Lagrangian cobordant to $(L_1', \ldots, L_m')$}.
\end{gathered}
\end{equation*}
Let $H \subset F\mathcal{L}ag^{s,m}(M)$ be a normal subgroup generated by the relations $R$. We define
\begin{equation*}
F\mathcal{C}ob^{s,m}(M) = F\mathcal{L}ag^{s,m}(M)/H.
\end{equation*}
The following question was stated by Biran and Cornea
\begin{center}
Is the group $F\mathcal{C}ob^{s,m}(M)$ commutative?
\end{center}
Note that $H$ contains much more elements than $R$. So, existence of the Lagrangians $L, L_1, L_2$, constructed in this paper, does not imply  that $L_1L_2 \neq L_2L_1$ in the group $F\mathcal{C}ob^{s,m}(M)$.

The Lagrangians constructed in this paper should help to study commutativity of $F\mathcal{C}ob^{s,m}(M)$. This question will be studied in a further paper, where we will prove that $L_1L_2 \neq L_2L_1$.

\vspace{0.08in}

We can define a similar free group for immersed Lagrangians, but we will get a group isomorphic to $Cob(M)$.

\vspace{.15in}

\noindent \textbf{Acknowledgments.}
The author thanks Viktor Ginzburg and Dominique Rathel-Fournier for helpful discussions.

\section{Preliminaries}

\subsection{Lagrangian submanifolds of $\mathbb{C}P^n$}

We consider $\mathbb{C}P^n$ endowed with the standard Fubini-Study form $\omega_{FS}$ such that $\omega_{FS}(\mathbb{C}P^1) = \pi$. There are two homomorphisms assigned to each embedded Lagrangian $L \subset \mathbb{C}P^n$
\begin{equation*}
\mu:  \pi_2(\mathbb{C}P^n, L) \rightarrow \mathbb{Z}, \quad \omega_{FS}:  \pi_2(\mathbb{C}P^n, L) \rightarrow \mathbb{R},
\end{equation*}
where $\mu$ is the Maslov class (see \cite[Section 2]{Smith} or \cite{Kail}) and $\omega_{FS}(\alpha) = \int_{\alpha}\omega_{FS}$. Define the minimal Maslov number
\begin{equation*}
N_L = min\{\mu(\alpha) > 0| \; \alpha \in \pi_2(\mathbb{C}P^n, L)\}, \;\;\; \text{$N_L = 0\;$ if $\;\mu \equiv 0$}.
\end{equation*}
Note that $N_L$ is even when $L$ is orientable. We say that $L$ is monotone if there exists a constant $\tau > 0$ such that
\begin{equation*}
\mu(\alpha) = \tau\omega_{FS}(\alpha), \;\;\; \text{for any $\alpha \in \pi_2(\mathbb{C}P^n, L)$ and $N_L \geqslant 2$}.
\end{equation*}
The constant $\tau$ is called the monotonicity constant. It can be shown (see \cite[Section 2]{Smith}) that $\tau = \frac{2(n+1)}{\pi}$ (if it exists).

We say that a Lagrangian $L$ is spin if it is orientable and the second Stiefel-Whitney class $w_2(L)$ vanishes.

Let $L_1, L_2 \subset \mathbb{C}P^n$ be orientable closed monotone Lagrangians.  Then, the minimal Maslov number of the pair $N_{L_1, L_2}$ is the greatest common divisor of $N_{L_1}$ and $N_{L_2}$.

The pair $(L_1, L_2)$ is called relatively spin if there exists an $h \in H^2(\mathbb{C}P^n; \mathbb{Z}_2)$ such that $h|_{L_1} = w_2(L_1)$ and $h|_{L_2} = w_2(L_2)$. Note that $h$ is defined up to elements of $H^2(\mathbb{C}P^n, L_1; \mathbb{Z}_2) \cap H^2(\mathbb{C}P^n, L_2; \mathbb{Z}_2)$. When we choose orientations of $L_1, L_2$ and $h$, then we say that we fix relative spin structure. Note that if $L_1, L_2$ are both spin, then the pair $(L_1, L_2)$ is always relatively spin.

We say that the intersection $S = L_1 \cap L_2$ is clean if $S$ is a manifold and $T_p L_1 \cap T_pL_2 = T_pS$ for any $p\in S$. For example, if $L_1 = L_2$, then their intersection is clean. Transversal intersections are also clean.

\begin{theorem}\label{basictheorem} (see \cite[Section 2.3]{Felix})
Assume that

\vspace{0.04in}
\noindent
- $L_1, L_2 \subset \mathbb{C}P^n$ are monotone Lagrangians and the minimal Maslov number of the pair $N_{L_1, L_2} > 3$

\vspace{0.04in}
\noindent
- $L_1, L_2$ are spin and we fix a relative spin structure

\vspace{0.04in}
\noindent
- Let $\Lambda = \mathbb{Z}[T, T^{-1}]$ be the algebra of Laurent polynomials, where we grade $deg(T)=-N_{L_1, L_2}$.

\vspace{0.05in}
Then, the Floer homology $HF_{*}(L_1, L_2; \Lambda)$ is defined. These homology groups depend on the relative spin structure of the pair $L_1, L_2$.

If $L_1 \cap L_2$ is clean and \textbf{connected}, then there exists a homological spectral sequence
\begin{equation*}
E_{p,q}^1 = \left\{
\begin{array}{l}
 H_{q}(L_1 \cap L_2; \mathbb{Z}), \quad \text{if $p \in  N_{L_1, L_2} \; \mathbb{Z}$,} \\
0 \quad \text{otherwise}
\end{array}
\right.
\end{equation*}
\begin{equation*}
\bigoplus_{p+q = *} E^{\infty}_{p,q} = HF_{*}(L_1, L_2; \Lambda).
\end{equation*}
It turns out that
\begin{equation*}
HF_{*}(L_1, L_2; \Lambda) = HF_{*+2}(L_1, L_2; \Lambda).
\end{equation*}
Let $N$ be a positive divisor of $N_{L_1, L_2}$ and $\widetilde{\Lambda} = \mathbb{Z}[T, T^{-1}]$, where $deg(T) = - N$. Then
\begin{equation*}
HF_{*}(L_1, L_2; \widetilde{\Lambda}) = \bigoplus\limits_{k=0}^{\frac{N_{L_1, L_2}}{N} - 1}  HF_{* + kN}(L_1, L_2; \Lambda).
\end{equation*}
\end{theorem}

\vspace*{0.07in}

Note that if $L_1 = L_2$, then the spectral sequence coincide with the spectral sequence of Oh.

The theorem above holds true in a more general setting and not only for Lagrangians of $\mathbb{C}P^n$. Also, the theorem holds true when $L_1 \cap L_2$ is clean, but not necessarily connected. However, when $L_1 \cap L_2$ is not connected we need to replace the spectral sequence by two other spectral sequences and more complicated coefficients. The reader can find all the details in \cite{Felix}.

In this paper we work only with connected clean intersections.

\subsection{Lagrangian cobordism}\label{Lagrdef}

In this section we give only basic definitions. Much more details can be found in \cite{Octav1}.

\vspace{0.08in}

Let $(M, \omega)$ be a closed symplectic manifold and $L_1, \ldots, L_k, L_1', \ldots, L_m' \subset M$ be  monotone embedded Lagrangians with the same monotonicity constants. In the previous section we discussed that if $(M, \omega) = (\mathbb{C}P^n, \omega_{FS})$, then monotonicity constant is the same for all monotone Lagrangians. Consider the complex plane $\mathbb{C}$ with the standard symplectic form $\omega_{\mathbb{C}}$ and $M \times \mathbb{C}$ endowed with $\omega \oplus \omega_{\mathbb{C}}$.

We say that there is a monotone  cobordism
\begin{equation*}
V: \; (L_1, \ldots, L_k) \rightsquigarrow (L_1', \ldots, L_m')
\end{equation*}
if there exists an embedded monotone Lagrangian $V \subset M \times \mathbb{C}$ such that the projection $\pi: V \rightarrow \mathbb{C}$ has the form shown in the following figure

  \begin{figure}[H]
  \centering
  \includegraphics[width=0.55\linewidth]{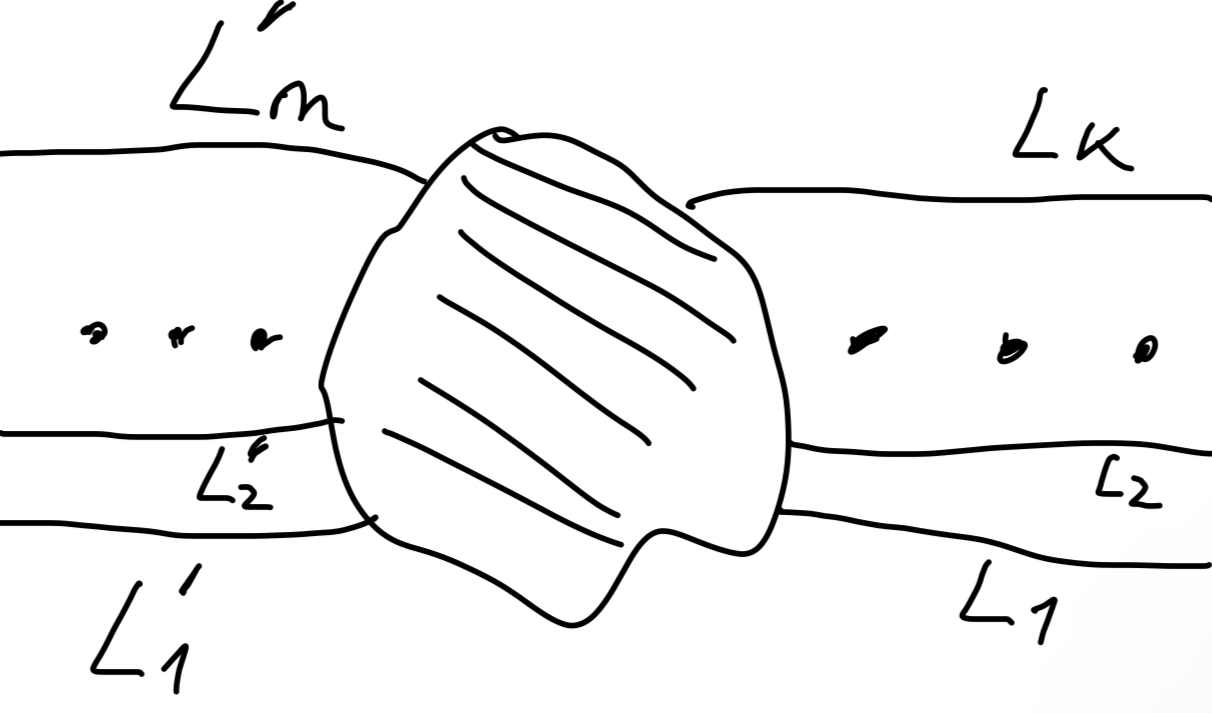}
\end{figure}

 In other words, for some compact set $K \subset \mathbb{C}$ we have
\begin{equation*}
\pi^{-1}(\mathbb{C} \setminus K) = \bigcup_{s=1}^m \big((-\infty, a) \times \{s\} \times L_s' \big) \bigcup_{r=1}^{k} \big( (b, +\infty) \times \{r\} \times L_r \big),
\end{equation*}
where $a, b$ are some real numbers.

We say that the cobordism is spin if all  Lagrangians $V, L_s, L_r'$ are spin. Let $j_s: L_s \hookrightarrow V$ and $j_r': L_r' \hookrightarrow V$ be embedding as ends, i.e.
\begin{equation*}
j_s(L_s) = \{c\} \times \{s\} \times L_s, \;\;\; j_r'(L_r') = \{c\} \times \{r\} \times L_r'
\end{equation*}
for some big $c \in \mathbb{R}$. Let $h$, $h_i$, $h_r'$ be a relative spin structures of $V$, $L_i$ and $L_r'$, respectively. We say that relative spin structures of $L_i$, $L_r'$ are induced from $V$ if
\begin{equation*}
j_s^{*}h = h_s, \quad (j_r')^{*}h = h_r'.
\end{equation*}

The following theorem shows that cobordant Lagrangians are closely related to each other.

\begin{theorem}\label{longexactseq} (\cite[Corollary 1.1.3]{BC2}, \cite[Section 4]{Haug})
Let $L, L_1, L_2, K$ be monotone spin Lagrangians with equal monotonicity constants. Assume that there exists a spin monotone cobordism $V: L \rightsquigarrow (L_1, L_2)$. If relative spin structures of $L, L_1, L_2, K$ are induced from $V$ and $N_K > 3$, then  we have the following exact sequence
\begin{equation*}
\rightarrow HF_{*}(K, L_2; \Lambda) \rightarrow HF_{*}(K, L_1; \Lambda) \rightarrow HF_{*}(K, L; \Lambda) \rightarrow HF_{*-1}(K, L_2; \Lambda) \rightarrow,
\end{equation*}
where $\Lambda = \mathbb{Z}[T, T^{-1}]$ and $deg(T)$ is the same for all Floer homology groups and is a common divisor of $N_{K, L_1}$, $N_{K, L_2}$, $N_{K, L}$, $N_V$.
\end{theorem}

Note that all the mentioned Lagrangians are orientable. This means that the minimal Maslov numbers are even and they all are divisible by $2$.

\subsection{Clean intersection and surgery}

Let $M$ be a closed symplectic manifold and $L_1, L_2$ be orientable monotone Lagrangian submanifolds.  Recall that the intersection $S = L_1 \cap L_2$ is called clean if $S$ is a manifold and $T_p L_1 \cap T_pL_2 = T_pS$ for any $p\in S$.

\begin{theorem}\label{surgerycobirdism}(see \cite[Section 2.2 and Section 6.1]{Cheuk})
Assume that $S = L_1 \cap L_2$ is clean and \textbf{connected}. Assume that $\pi_1(M) = 0$. Then we can do a surgery and get a monotone Lagrangian $L = L_1 \#_S L_2 \subset M$ such that there is a monotone cobordism
\begin{equation*}
V: L \rightsquigarrow (L_1, L_2)
\end{equation*}
\end{theorem}

\vspace*{0.06in}

Let us note the following

\vspace{0.04in}
\noindent
- By definition, $V \subset M \times \mathbb{C}$ and the projection $V \rightarrow \mathbb{C}$ has the form shown in the following figure

  \begin{figure}[H]
  \centering
  \includegraphics[width=0.55\linewidth]{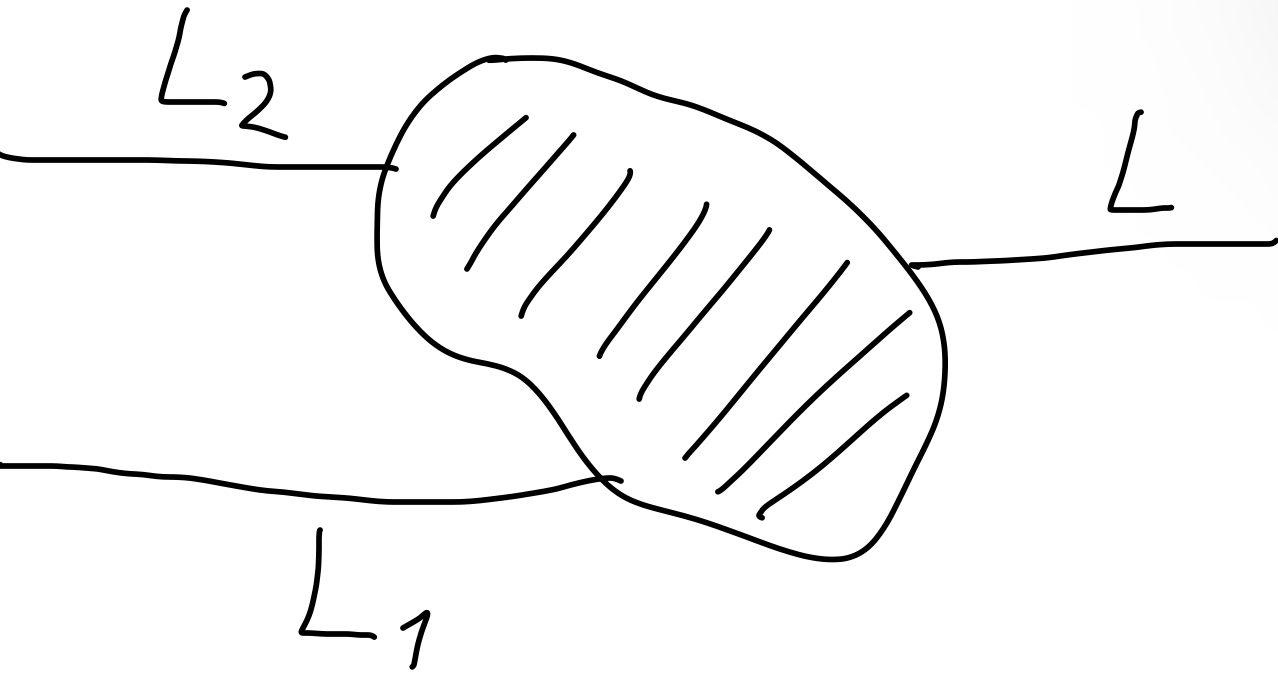}
\end{figure}

\vspace{0.04in}
\noindent
- The condition that $S$ is connected is crucial.

\vspace{0.04in}
\noindent
- $V$ is homotopy equivalent to $L_1 \cup L_2$, i.e. $V$ is homotopy equivalent to $L_1$ attached to $L_2$ along $S$.

\vspace{0.04in}
\noindent
- Theorem $\ref{longexactseq}$ says that there is a long exact sequence related to Lagrangians $L_1, L_2, L_1 \#_S L_2$.

\vspace{0.04in}
\noindent
- $L_1 \#_S L_2$ may not be Hamiltonian isotopic to $L_2 \#_S L_1$.

\vspace{0.08in}

The reader can find more details about Lagrangian surgery in \cite{Hicks}.

\section{Construction of noncommutative Lagrangian cobordisms}

Let $L_2 = \mathbb{R}P^{7} \subset \mathbb{C}P^{7}$. In this section we construct a monotone spin Lagrangian $L_1 \subset \mathbb{C}P^{7}$ and a monotone spin Lagrangian cobordism $V: L \rightsquigarrow (L_1, L_2)$. Then, we prove that there is no monotone spin Lagrangian cobordism between $L$ and $(L_2, L_1)$. The construction consists of the following steps:

\vspace{0.07in}
\noindent
- We construct $L_1$ and show that the intersection $S = L_1 \cap L_2$ is clean and connected. Then, Theorem $\ref{surgerycobirdism}$ says that we can do the surgery and get a monotone Lagrangian $L = L_1 \#_S L_2$ and a monotone Lagrangian cobordism $V: L \rightsquigarrow (L_1, L_2)$.

\vspace{0.05in}
\noindent
- We prove that the Lagrangians $V, L, L_1, L_2$ are spin.

\vspace{0.05in}
\noindent
- Using the cobordism $V: L \rightsquigarrow (L_1, L_2)$ and the long exact sequence from Theorem $\ref{longexactseq}$ we prove that there is no spin monotone cobordism from $L$ to $(L_2, L_1)$.

\subsection{Construction of the Lagrangian $L_1$}\label{mainLagr}

Consider a submanifold $\mathcal{R} \subset \mathbb{R}P^{7}$  defined by a quadric

\begin{equation}\label{equationemb}
u_1^2 + u_2^2 + u_3^2 + u_4^2  - u_5^2 - u_6^8 - u_7^2 - u_8^2 = 0
\end{equation}
and $\widetilde{\mathcal{R}} \subset \mathbb{R}^{8}$ defined by a system
\begin{equation*}
\left\{
 \begin{array}{l}
u_1^2 + u_2^2 + u_3^2 + u_4^2 + u_5^2 + u_6^2 + u_7^2 + u_8^2 = 1 \\
u_1^2 + u_2^2 + u_3^2 + u_4^2  - u_5^2 - u_6^8 - u_7^2 - u_8^2 = 0
 \end{array}
\right.
\end{equation*}
We see that $\widetilde{\mathcal{R}} = S^{3} \times S^3$ and $\mathcal{R} = (S^{3} \times S^3)/\mathbb{Z}_2$, where we identify antipodal points. It is known that
\begin{equation*}
\mathcal{R} = (S^{3} \times S^3)/\mathbb{Z}_2 = SO(4) = SO(3) \times S^3 = \mathbb{R}P^3 \times S^3.
\end{equation*}
Let $S^1 = \{e^{i\pi\varphi}|\; 0 \leqslant \varphi \leqslant 2 \}$ be the unit circle.  Consider $\mathcal{R} \times S^1$ and a map
\begin{equation}\label{embformula}
\begin{gathered}
\psi: \mathcal{R} \times S^1 \rightarrow \mathbb{C}P^{3}, \\
\psi(u_1, u_2, u_3, u_4, u_5, u_6, u_7, u_8, \varphi) = \\
 [u_1e^{\pi\varphi}: u_2e^{\pi\varphi} : u_3e^{\pi\varphi}: u_4e^{\pi\varphi}: u_5: u_6: u_7: u_8].
\end{gathered}
\end{equation}
In fact, $\psi$ is not embedding because
\begin{equation*}
\psi(u_1, u_2, u_3, u_4, u_5, u_6, u_7, u_8, \varphi) = \psi(-u_1, -u_2, -u_3, -u_4, u_6, u_6, u_7, u_8, \varphi + 1)
\end{equation*}
Define an action of $\mathbb{Z}_2$ on $\mathcal{R} \times S^1$ by the following formula
\begin{equation*}
\begin{gathered}
\varepsilon\cdot (u_1, u_2, u_3, u_4, u_5, u_6, u_7, u_8, \varphi) = \\
((-1)^{\varepsilon}u_1, (-1)^{\varepsilon}u_{2}, (-1)^{\varepsilon}u_3, (-1)^{\varepsilon}u_{4}, u_5, u_6, u_7, u_8, \varphi + 1), \;\;\; \varepsilon \in \mathbb{Z}_2.
\end{gathered}
\end{equation*}
We see that
\begin{equation*}
\psi(\varepsilon\cdot (u_1, u_2, u_3, u_4, u_5, u_6, u_7, u_8, \varphi)) = \psi(u_1, u_2, u_3, u_4, u_5, u_6, u_7, u_8, \varphi)
\end{equation*}
and  the action of  $\mathbb{Z}_2$ is free on the second factor of $\mathcal{R} \times S^1$. This means that we have a fibration
\begin{equation*}
(\mathcal{R} \times S^1)/\mathbb{Z}_2 \xrightarrow{\mathcal{R}} S^1/\mathbb{Z}_2 = \mathbb{R}P^1 =  S^1.
\end{equation*}

Denote
\begin{equation*}
L_1 = \psi\big((\mathcal{R} \times S^1)/\mathbb{Z}_2\big) \subset \mathbb{C}P^{7},
\end{equation*}
where $\psi$ is defined on the quotient by the same formula $(\ref{embformula})$.

\begin{theorem} (see \cite[Theorem 1.3]{Vardan})
We have

\vspace{0.04in}
\noindent
- $L_1$ is diffeomorphic to $\mathcal{R} \times S^1 = \mathcal{R}P^3 \times S^3 \times S^1$. In particular $L_1$ is spin.

\vspace{0.04in}
\noindent
- $L_1$ is monotone embedded Lagrangian of $\mathbb{C}P^7$ with minimal Maslov number $N_L = 4$.
\end{theorem}

Let us prove the following simple lemma.

\begin{lemma}\label{cleanintersection}
Consider the standard $\mathbb{R}P^{7} \subset \mathbb{C}P^{7}$. We have $L_1 \cap \mathbb{R}P^{7} = \mathcal{R}$ and $L_1 \cap \mathbb{R}P^{7}$ is clean.
\end{lemma}
\begin{proof}
The fact that $L_1 \cap \mathbb{R}P^{7} = \mathcal{R}$ follows directly from map $(\ref{embformula})$. Consider $p \in L_1 \cap \mathbb{R}P^{7}$ and a small neighborhood $U \subset L_1$  of $p$. The Lagrangian $L_1$ in $U$ can be parameterized in the following way
\begin{equation*}
L_1 = [f_1e^{i\pi\varphi}: f_2e^{i\pi\varphi}: f_3e^{i\pi\varphi}: f_4e^{i\pi\varphi}: f_5: f_6: f_7: f_8], \\
\end{equation*}
where $f_j$ are real-valued functions. This shows that $T_pL_1 \cap T_p\mathbb{R}P^{7} = T_p\mathcal{R}$.
\end{proof}

\subsection{Construction and properties of the cobordism}\label{cobconstr}

In the previous section we constructed  $L_1 = \mathcal{R} \times S^1$. Denote $L_2 = \mathbb{R}P^{7}$. Lemma \ref{cleanintersection} says that $L_1 \cap L_2 = \mathcal{R}$ and  is clean. We get from Theorem $\ref{surgerycobirdism}$ that there exists a monotone Lagrangian and a cobordism
\begin{equation*}
L = L_1 \#_{\mathcal{R}} L_2, \quad V: L \rightsquigarrow (L_1, L_2).
\end{equation*}
Let $j_1: L_1 \hookrightarrow V$ and $j_2: L_2 \hookrightarrow V$ be embedding as ends, i.e.
\begin{equation*}
j_1(L_1) = \{c\} \times \{1\} \times L_1, \;\;\; j_2(L_2) = \{c\} \times \{2\} \times L_2
\end{equation*}
for some big $c \in \mathbb{R}$. We also have an embedding $j$ of the disjoint union
\begin{equation*}
j: L_1 \coprod L_2 \hookrightarrow V, \;\;\; j(L_1 \coprod L_2) = (\{c\} \times \{1\} \times L_1) \coprod (\{c\} \times \{2\} \times L_2).
\end{equation*}

Let $i_1: \mathcal{R} \rightarrow L_1$ be the embedding as the first factor and $i_2: \mathcal{R} \rightarrow L_2$ be the embedding as the set of solutions of the quadric.

\begin{theorem}
The Lagrangians $V$ and $L$ are orientable and spin.
\end{theorem}

\begin{proof}
Let $U \subset V$ be a small neighborhood of $\{c\} \times \{1\} \times L_1$ in $V$. By definition $U = I \times L_1$, where $I$ is some small interval. The same is true about some small neighborhood of $\{c\} \times \{2\} \times L_2$. This means that
\begin{equation*}
j_1^{*}TV = TL_1 \times \mathbb{R}, \;\;\; j_2^{*}TV = TL_2 \times \mathbb{R}.
\end{equation*}

Let $w_1(V)$, $\omega_2(V)$ be the Stiefel-Whitney classes of $V$. The arguments above show that $\omega_k(j_i^{*}TV) = \omega_k(L_i)$ for $i,k = 1, 2$.  Since $L_1$, $L_2$ are orientable and spin we have
\begin{equation*}
j^{*}w_k(V) = w_k(L_1) + w_k(L_2) = 0, \quad k=1,2.
\end{equation*}

\begin{lemma}
The pullback $j^{*}$ is injective on $H^{k}(V; \mathbb{Z}_2)$ for  $k=1, 2$.
\end{lemma}

\begin{proof}
 Since $V$ is homotopy equivalent to $L_1 \cup L_2$ and $L_1 \cap L_2 = \mathcal{R}$, we have the exact sequence of Mayer-Vietoris 
\begin{equation*}
\rightarrow H^k(V; \mathbb{Z}_2) \xrightarrow{(j_1^{*}, j_2^{*})} H^k(L_1; \mathbb{Z}_2) \oplus H^k(L_2; \mathbb{Z}_2) \xrightarrow{i_1^{*} - i_2^{*}} H^k(\mathcal{R}; \mathbb{Z}_2) \rightarrow
\end{equation*}
Since $i_1^{*}$ is surjective, we obtain that the map $(j_1^{*}, j_2^{*})$ is injective. This means that $j^{*}$ is injective.
\end{proof}
Injectivity of $j^{*}$ and $j^{*}w_k(V) = 0$ imply that $w_k(V) = 0$ for $k=0,1$.

\vspace{0.08in}

Let $h: L \rightarrow V$ be the embedding of $L$ as an end. The same arguments show that $h^{*}TV = TL \times \mathbb{R}$. Therefore, $\omega_k(L) = h^{*}\omega_k(V) = 0$, where $k=1,2$.
\end{proof}

\subsection{Nonexistence of the cobordism $L \rightsquigarrow (L_2, L_1)$}

Recall that we denote $\mathbb{R}P^7$ by $L_2$ and $N_{L_2} = 8$.

\begin{lemma}\label{floerproj}
Let $\Lambda' = \mathbb{Z}[T, T^{-1}]$, where $deg(T) = -8$. For any relative spin structure of $L_2$ and any $k$ we have
\begin{equation*}
HF_{2k}(L_2, L_2; \Lambda') = 0, \quad HF_{2k-1}(L_2, L_2; \Lambda') = \mathbb{Z}_2.
\end{equation*}
\end{lemma}

\begin{proof}
Since intersection of $L_2$ with itself is clean and $N_{L_2, L_2} = N_{L_2} = 8$, we have the homological spectral sequence from Theorem $\ref{basictheorem}$.

Theorem $\ref{basictheorem}$ says that $HF_m(L_2, L_2; \Lambda') = HF_{m+2}(L_2, L_2; \Lambda')$ for any $m$. So, it is sufficient to find $HF_1(L_2, L_2; \Lambda')$ and $HF_2(L_2, L_2; \Lambda')$. Note the only nontrivial columns are $E_{8k, *}^1$, where $k$ is arbitrary. This means that
\begin{equation*}
\begin{gathered}
0 = E^1_{0, 2} = E^{\infty}_{0, 2} \;\; \Rightarrow \;\; HF_2(L_2, L_2; \Lambda') = 0, \\
\mathbb{Z}_2 = E_{0,1}^{1} = E_{0,1}^{\infty} = HF_1(L_2, L_2; \Lambda').
\end{gathered}
\end{equation*}
We see that $HF_0(L_2, L_2; \Lambda') = 0$ and $HF_7(L_2, L_2; \Lambda') = \mathbb{Z}_2$. Let $d^8$ be the differential of the page $E^8$. We get that
\begin{equation*}
d^8: \; E_{8,0}^8 = \mathbb{Z} \rightarrow E_{0, 7}^8 = \mathbb{Z}
\end{equation*} 
is nontrivial and $d^8(1) = 2$. Change of relative spin structure changes the sign of $d^8(1)$, but this does not change the Floer homology groups.
\end{proof}

In Section $\ref{mainLagr}$ we constructed a monotone spin Lagrangian of $L_1 = \mathbb{R}P^3 \times S^3 \times S^1 \subset \mathbb{C}P^7$ and showed that
\begin{equation*}
L_1 \cap L_2 = \mathcal{R} = \mathbb{R}P^3 \times S^3.
\end{equation*}
Since $N_{L_1} = 4$ and $N_{L_2} = 8$, we get that $N_{L_1, L_2} = 4$.

\begin{lemma}\label{inthomlemma}
Let $\Lambda'' = \mathbb{Z}[T, T^{-1}]$  and $deg(T) = -4$. For any relative spin structure we have two options
\begin{equation*}
\begin{gathered}
HF_0(L_2, L_1; \Lambda'') = HF_1(L_2, L_1; \Lambda'') = 0 \quad \text{or} \\
HF_0(L_2, L_1; \Lambda'') =  HF_{1}(L_2, L_1; \Lambda'') = \mathbb{Z}_2.
\end{gathered}
\end{equation*}
\end{lemma}

\begin{proof}
Note that the only nontrivial columns are $E_{4k, *}^1$, where $k$ is arbitrary. Therefore, $E^1 = E^4$ and $E^5 = E^{\infty}$. Let $d^4$ be the differential of $E^4$.

\vspace{0.05in}
\textbf{1.} Assume that
\begin{equation*}
d^4: \;\; E_{0, 1}^4 = \mathbb{Z}_2 \rightarrow E_{-4, 4}^4 = \mathbb{Z}_2
\end{equation*}
is not trivial. Hence $E_{0,1}^5 = E_{-4,4}^5 = 0$. Since $E^1_{-4, 5} = H_5(\mathcal{R}; \mathbb{Z}) = 0 = E_{-4, 5}^5$, we obtain
\begin{equation*}
E^5_{0, 1} \oplus E^5_{-4, 5} = 0 =  HF_1(L_2, L_1; \Lambda'').
\end{equation*}
Theorem $\ref{basictheorem}$ says that
\begin{equation*}
0 = HF_1(L_2, L_1; \Lambda'') = HF_{-1}(L_2, L_1; \Lambda'') = E^{\infty}_{-4,3} = E^{5}_{-4,3}  .
\end{equation*}
Since $E^4_{-4,3} = H_3(\mathcal{R}; \mathbb{Z}) = \mathbb{Z} \oplus \mathbb{Z}$, $E^4_{0,0} = H_0(\mathcal{R}; \mathbb{Z}) = \mathbb{Z}$ and  $E^4_{-8, 6} = H_6(\mathcal{R}; \mathbb{Z}) = \mathbb{Z}$, we see that to kill $E^4_{-4, 3}$ we need both  $d^4: E^4_{0,0} \rightarrow E^4_{-4, 3}$ and $d^4: E^4_{-4,3} \rightarrow E^4_{-8, 6}$ to be nontrivial. Therefore,
\begin{equation*}
HF_0(L_2, L_1; \Lambda'') = E^5_{0,0} \oplus E^5_{-4, 4} = 0.
\end{equation*}

\vspace*{0.07in}
\textbf{2.} Assume that
\begin{equation*}
d^4: \;\; E_{0, 1}^4 = \mathbb{Z}_2 \rightarrow E_{-4, 4}^4 = \mathbb{Z}_2
\end{equation*}
is trivial. This means that $E_{0, 1}^5 = E_{-4,4}^5 = \mathbb{Z}_2$.  Arguing as before, we get $E_{-4, 5}^1 = H_5(\mathcal{R}; \mathbb{Z}) = 0  = E_{-4,5}^5$ and
\begin{equation*}
\begin{gathered}
E^5_{0, 1} \oplus E^5_{-4, 5} = \mathbb{Z}_2 =  HF_1(L_2, L_1; \Lambda''), \\
\mathbb{Z}_2 = HF_{1}(L_2, L_1; \Lambda'') = HF_{-1}(L_2, L_1; \Lambda'') = E^{\infty}_{-4,3} = E^{5}_{-4,3}.
\end{gathered}
\end{equation*}
We know that $E^1_{-4,3} = H_3(\mathcal{R}; \mathbb{Z}) = \mathbb{Z} \oplus \mathbb{Z}$. If $E^5_{-4,3} = \mathbb{Z}_2$, then both $d_4: E_{0,0}^4 = \mathbb{Z} \rightarrow E_{-4,3}^4$ and $d^4: E_{-4,3}^4 \rightarrow E_{-8, 6}^4 = \mathbb{Z}$ are nonzero. Hence $E^5_{0,0} = 0$ and
\begin{equation*}
HF_0(L_2, L_1; \Lambda'') = E^5_{0,0} \oplus E^5_{-4, 4} = \mathbb{Z}_2.
\end{equation*}
\end{proof}

In Section $\ref{cobconstr}$ we constructed the spin monotone cobordism $V: L \rightsquigarrow (L_1, L_2)$. Assume that there exists a spin monotone Lagrangian cobordism
\begin{equation*}
\widetilde{V}: L \rightsquigarrow (L_2, L_1).
\end{equation*}
Since all the Lagrangians and cobordisms are orientable, we get that all the minimal Maslov numbers are even.

Consider $\Lambda = \mathbb{Z}[T, T^{-1}]$, where $deg(T) = -2$. Theorem $\ref{basictheorem}$, Lemma $\ref{floerproj}$, and Lemma $\ref{inthomlemma}$ imply that
\begin{equation*}
HF_0(L_2, L_2; \Lambda) = 0, \quad HF_1(L_2, L_2; \Lambda) = \mathbb{Z}_2^4
\end{equation*}
and for any relative spin structure we have two options
\begin{equation*}
\begin{gathered}
HF_0(L_2, L_1; \Lambda) = HF_1(L_2, L_1; \Lambda) = 0 \quad \text{or} \\
HF_0(L_2, L_1; \Lambda) =  HF_{1}(L_2, L_1; \Lambda) = \mathbb{Z}_2^2.
\end{gathered}
\end{equation*}
Theorem $\ref{longexactseq}$ says that we have the related long exact sequences for the cobordism $V$ and $\widetilde{V}$. Let $K = L_2$ in the exact sequences. The cobordism $V$ gives us
\begin{equation}\label{firstexact}
\begin{gathered}
HF_1(L_2, L_2; \Lambda) \rightarrow HF_1(L_2, L_1; \Lambda) \rightarrow HF_1(L_2, L; \Lambda) \rightarrow  \\
HF_0(L_2, L_2; \Lambda) \rightarrow HF_0(L_2, L_1; \Lambda)
\end{gathered}
\end{equation}
and the cobordism $\widetilde{V}$ gives us
\begin{equation}\label{secondexact}
\begin{gathered}
HF_1(L_2, L_1; \Lambda) \rightarrow HF_1(L_2, L_2; \Lambda) \rightarrow HF_1(L_2, L; \Lambda) \rightarrow \\
HF_0(L_2, L_1; \Lambda) \rightarrow HF_0(L_2, L_2; \Lambda)
\end{gathered}
\end{equation}
Lemma $\ref{inthomlemma}$ says that we have two options.

\vspace{0.06in}
\noindent
\textbf{1.} Suppose that $HF_0(L_2, L_1 \Lambda) = HF_1(L_2, L_1; \Lambda) = 0$. Then, the long exact sequences $(\ref{firstexact})$ and $(\ref{secondexact})$ have the form
\begin{equation*}
\begin{gathered}
0 \rightarrow  HF_1(L_2, L; \Lambda) \rightarrow 0 \\
0 \rightarrow \mathbb{Z}_2^4 \rightarrow HF_1(L_2, L; \Lambda) \rightarrow 0
\end{gathered}
\end{equation*}
We obviously get a contradiction.

\vspace{0.06in}
\noindent
\textbf{2.} Suppose that $HF_0(L_2, L_1; \Lambda) = HF_1(L_2, L_1; \Lambda) = \mathbb{Z}_2^2$. Then, the long exact sequences $(\ref{firstexact})$ and $(\ref{secondexact})$ have the form
\begin{equation*}
\begin{gathered}
\mathbb{Z}_2^4  \rightarrow \mathbb{Z}_2^2 \rightarrow HF_1(L_2, L; \Lambda) \rightarrow  0 \\
\mathbb{Z}_2^2 \rightarrow \mathbb{Z}_2^4 \rightarrow HF_1(L_2, L; \Lambda) \rightarrow \mathbb{Z}_2^2 \rightarrow 0
\end{gathered}
\end{equation*}
The first exact sequence says that there is a surjection from $\mathbb{Z}_2^2$ to $HF_1(L_2, L; \Lambda)$. The second exact sequence says that there is a surjection from $HF_1(L_2, L; \Lambda)$ to $\mathbb{Z}_2^2$. Therefore, $HF_1(L_2, L; \Lambda)$ consists of four elements.

These facts show that the map $HF_1(L_2, L; \Lambda) \rightarrow \mathbb{Z}_2^2$ in the second sequence has trivial kernel. Hence, the map $\mathbb{Z}_2^2 \rightarrow \mathbb{Z}_2^4 $ from the second sequence is surjective. This is obviously not possible.

\vspace{0.06in}

As a result, we see that the cobordism $\widetilde{V}: L \rightsquigarrow (L_2, L_1)$ does not exist.

\end{document}